\documentclass[12pt]{amsart}
\usepackage[margin=2cm]{geometry}
\usepackage[utf8]{inputenc}
\usepackage[english]{babel}
\usepackage{amsthm}
\usepackage{mathtools}
\usepackage{tikz-cd}
\usepackage[style=alphabetic,maxnames=10,backend=bibtex,hyperref=true]{biblatex}
\newtheorem{theorem}{Theorem}
\newtheorem{cor}[theorem]{Corollary}
\newtheorem{lemma}[theorem]{Lemma}
\newtheorem{defi}[theorem]{Definition}
\usepackage{amsmath}
\usepackage{amssymb}
\usepackage{graphicx}
\usepackage[hidelinks]{hyperref}
\newcommand{\UU}{\mathcal{U}}
\newcommand{\id}{\mathsf{id}}
\newcommand{\pt}{\mathrm{pt}}
\newcommand{\refl}{\mathsf{refl}}
\newcommand{\Sq}{\mathrm{Sq}}
\newcommand{\ap}{\mathrm{ap}}
\newcommand{\Z}{\mathbb{Z}}
\newcommand{\grp}{\mathrm{grp}}
\newcommand{\K}{\mathrm{K}}
\newcommand{\Zc}{\mathrm{Z}}
\DeclareMathOperator{\baut}{BAut}
\newcommand*\sq{\mathbin{\vcenter{\hbox{\rule{.3ex}{.3ex}}}}}
\DeclarePairedDelimiter{\trunc}{\lVert}{\rVert}

\title{Eilenberg--MacLane spaces and stabilisation
in~homotopy~type~theory}
\author{David Wärn}
\email{warnd@chalmers.se}

\begin{document}
\maketitle

\begin{abstract}

In this note, we study the delooping of spaces and maps in homotopy type
theory. We show that in some cases, spaces have a unique delooping, and give a
simple description of the delooping in these cases. We explain why some maps,
such as group homomorphisms, have a unique delooping. We discuss some
applications to Eilenberg--MacLane spaces and cohomology. 

\end{abstract}

\section{Introduction}

The loop space functor $\Omega$ is an operation on pointed types and pointed
maps between them. In this note, we study the \emph{delooping} of types and
maps: given a pointed type $X$, when can we find a pointed type whose loop
space is equivalent to $X$? And given a pointed map $f : \Omega A \to_\pt
\Omega B$, when can we find a map $A \to_\pt B$ whose looping equals $f$? The
general answer is rather complicated, involving group operations and an
infinite tower of coherences, but according to the stabilisation theorem
\cite{highergroups}, the answer becomes much simpler if we put some
connectivity and truncation assumptions on $A$ and $B$. The purpose of this
note is to give a direct, type-theoretic account of these simple special cases.
We also explain how to use these results to set up the theory of
Eilenberg--MacLane spaces and cohomology operations. We assume only basic
familiarity with homotopy type theory, as developed in \cite{hottbook}. We will
not need to assume the Freudenthal suspension theorem, nor will we make use of
any higher inductive types other than propositional truncation.

\subsection*{Notation}

As in \cite{hottbook}, we write $a = b$ for the type of identifications between
$a$ and $b$, $\refl_a : a = a$ for the reflexivity
identification, $\sq : (a = b) \to (b = c) \to (a = c)$ for path concatenation,
$\ap_f : (a = b) \to (f\, a = f\, b)$ for the action of a function on paths,
$\UU$ for a univalent universe, and $\trunc A$ for propositional
truncation. We write $(a : A) \to B\, a$ for the $\Pi$-type $\Pi_{a : A}
B\, a$, and $(a : A) \times B\, a$ for the $\Sigma$-type $\Sigma_{a : A}
B\, a$. We write $\UU_\pt$ for the type $(X : \UU) \times X$ of
\emph{pointed} types. For $A : \UU_\pt$, we will write $|A| : \UU$ for its
underlying type, and $\pt_A : |A|$ for its point. For $A, B : \UU_\pt$, we
write $A \to_\pt B$ for the type $(f : |A| \to |B|) \times (f\, \pt_A =
		\pt_B)$ of \emph{pointed} functions. For $f : A \to_\pt B$, we
write $|f| : |A| \to |B|$ for the underlying function, and $\pt_f : |f|\,
\pt_A = \pt_B$ for the proof that it is pointed. For $A : \UU_\pt$, we
write $\Omega A : \UU_\pt$ for the loop space $(\pt_A = \pt_A, \refl_{\pt_A})$.
For $f : A \to_\pt B$ we write $\Omega f : \Omega A \to_\pt \Omega B$ for the
action on loops, $p : \pt_A = \pt_A \mapsto \pt_f^{-1} \sq \ap_{|f|} p \sq \pt_f : \pt_B = \pt_B$.
We write $A \simeq_\pt B$ for the type $(f : A \simeq B) \times f\, \pt_A = \pt_B$ of pointed
equivalences.

\subsection*{Acknowledgements}

We thank Thierry Coquand for his support throughout the
project, as well as Felix Cherubini, Louise Leclerc, Jarl G. Taxerås Flaten,
and Axel Ljungström for fruitful discussions.

\section{Delooping types}

Let $X : \UU_\pt$ be a pointed type, and suppose we want -- without further
inputs -- to construct a delooping of $X$. That is, we want to find a pointed
type whose loop space is equivalent to $X$. One way would be to use the
\emph{suspension} $\Sigma X$ \cite{eilenberg}, which is freely generated by a map $X \to_\pt
\Omega \Sigma X$ and so necessarily maps \emph{to} any delooping of $X$.
Instead, we will use a \emph{cofree} construction, which necessarily has a map
\emph{from} any delooping of $X$. Similar ideas are discussed in \cite{central}.

\begin{defi}
For $X : \UU_\pt$, the type $TX$ of \emph{$X$-torsors} is given by
\[
	TX \coloneqq (Y : \UU) \times \trunc Y \times (y : Y) \to X \simeq_\pt (Y, y).%
\footnote{A priori, since $\UU$ is a \emph{large} type, so is $TX$. 
However, we could just as
well quantify over $Y : \baut |X|$ in the definition of $TX$, where 
$\baut |X| \simeq (Y : \UU) \times \trunc{Y \simeq |X|}$. It is reasonable to assume that $\baut |X|$
is small, either by the replacement principle from \cite{join}, or by
simply postulating the existence of enough small univalent type families. 
In the rest of the note we ignore universe issues and assume $TX : \UU$.}
\]
\end{defi}

Intuitively, an $X$-torsor is a type which looks like $X$ at every point, and
merely has a point, even though we might not have access to any particular point.

\begin{theorem}\label{thm}
If the type $((Y,h,\mu) : TX) \times Y$ of pointed torsors is contractible,
then $TX$ is a delooping of $X$.
That is, we have a point $\pt_{TX} : TX$ with an equivalence $\Omega(TX, \pt_{TX}) \simeq_\pt X$.
Moreover, $TX$ is the unique delooping of $X$ in this case.
\end{theorem}
\begin{proof}
For the first part, we apply the fundamental theorem of identity types~\cite[Theorem~11.2.2]{intro} 
to the type family over
$TX$ given by $(Y,h,\mu) \mapsto Y$. Say $(Y,h,\mu) : TX$ and $y : Y$.
We then point $TX$ by $\pt_{TX} \coloneqq (Y,h,\mu)$. Note that $X \simeq_\pt (Y,y)$
by $\mu(y)$.
The fundamental theorem tells us that $(Y,h,\mu) = (Y',h',\mu')$ is equivalent
to $Y'$ for any $(Y',h',\mu') : TX$, where the map from $(Y,h,\mu) = (Y',h',\mu')$
to $Y'$ is given by transporting $y$. That is, saying a torsor is trivial
is equivalent to saying that it is pointed. In particular 
$(\pt_{TX} = \pt_{TX}) \simeq_\pt (Y,y) \simeq_\pt X$ as claimed.

We now show uniqueness. Consider another delooping $Z : \UU_\pt$, $e : \Omega Z \simeq_\pt X$
with $Z$ connected. We first define a map $f : |Z| \to TX$. For $z : |Z|$, we take the underlying
type of $f\, z$ to be $z = \pt_Z$. This is merely inhabited since $Z$
is connected, and for any $p : z = \pt_Z$ we have $(z = \pt_Z,p) \simeq_\pt \Omega Z \simeq_\pt X$
by induction on $p$. This finishes the definition of $f$. We have $f\,\pt_Z = \pt_{TX}$
since $f\, \pt_Z$ is pointed by $\refl_{\pt_Z}$ and hence trivial.
The action of $f$ on paths $(z = \pt_Z) \to (f\,z = \pt_{TX})$
must send $p : z = \pt_Z$ to the proof $f\,z = \pt_{TX}$ corresponding to the point $p$
of $f\,z$, by induction $p$. By unfolding definitions it can be seen that
the action $\Omega Z \to_\pt \Omega TX$ on loops corresponds to the identity $X \to_\pt X$.
In particular it is an equivalence. By Whitehead's principle~\cite[Corollary~8.8.2]{hottbook},
$f$ itself is an equivalence. By univalence, the delooping $(Z,e)$ equals the
one given by $TX$.
\end{proof}
The following lemma provides an alternative description of the type of pointed 
$X$-torsors, which will make it feasible to determine when it is contractible.

\begin{lemma}\label{mainequiv}
We have an equivalence of types 
\[ ((Y, h, \mu) : TX) \times Y
\simeq (\mu : (x : |X|) \to X \simeq_\pt (|X|, x)) \times (\mu\, \pt_X = \id_X).
\]
\end{lemma}
The right-hand side is roughly the type of \emph{coherent $H$-space structures} on $X$,
but note that we asymmetrically require invertibility on one side.
\begin{proof}
We have
\begin{align*}
((Y, h, \mu) : TX) \times Y
	&\simeq 
	(Y : \UU) \times \trunc Y \times ((y : Y) \to X \simeq_\pt (Y, y)) \times Y
\\
	&\simeq 
	(Z : \UU_\pt) \times (\mu : (z : |Z|) \to X \simeq_\pt (|Z|, z))
\\
	&\simeq 
				(Z : \UU_\pt) \times 
				(\mu : (z : |Z|) \to X \simeq_\pt (|Z|, z)) \times
				(p : X \simeq_\pt Z) \times 
				(\mu\, \pt_Z = p)\\
&\simeq (\mu : (x : |X|) \to X \simeq_\pt (|X|, x)) \times (\mu\, \pt_X = \id_X).
\end{align*}
In the first line, we simply unfold the definition of $TX$, and in the second
line we do some simple rearrangement, dropping the redundant assumption $\trunc
Y$. In the third line, we use contractibility of singletons to add two
redundant fields $p : X \simeq_\pt Z$ and $\mu\, \pt_Z = p$. And in the final
line, we use univalence and contractibility of singletons to remove two
redundant fields $Z$ and $p$.
\end{proof}

The following lemma will be our main tool to determine when types are
contractible. It is a special case of Lemma 8.6.1 from \cite{hottbook}, and
has a direct proof by induction.

\begin{lemma}\label{pointedsection}
If $A : \UU_\pt$ is an $n$-connected%
\footnote{While there are several equivalent definitions of connectedness, this
note is most easily understood with a \emph{recursive} definition: every type
is $(-2)$-connected, and a type is $(n+1)$-connected if it is merely inhabited and
its identity types are $n$-connected.}
pointed type, $B : |A| \to \UU$ is a family
of $(n+m+1)$-truncated types, and $\pt_B : B\, \pt_A$, then the type of `pointed sections of $B$',
\[
	(f : (a : |A|) \to B\, a) \times (f\, \pt_A = \pt_B),
\]
is $m$-truncated.
\end{lemma}

\begin{cor}\label{nondep}
If $A : \UU_\pt$ is $n$-connected and $B : \UU_\pt$ is $(n+m+1)$-truncated, then
$A \to_\pt B$ is $m$-truncated. If $A$ and $B$ are \emph{both} $n$-connected and
$(n+m+1)$-truncated, then $A \simeq_\pt B$ is also $m$-truncated.
\end{cor}
\begin{proof}
The first claim is a direct consequence of Lemma \ref{pointedsection}. For the
second, we have an equivalence between $A \simeq_\pt B$ and the type 
$(f : A \to_\pt B) \times (g\, h : B \to_\pt A) \times (f \circ g = \id_B)
\times (h \circ f = \id_A)$ of biinvertible pointed maps. This is 
$m$-truncated since $m$-truncated types are closed under $\Sigma$ and identity types.
\end{proof}

\begin{cor}\label{tors}
If $X$ is $n$-connected and $(2n+m+2)$-truncated, then the type of pointed
$X$-torsors is $m$-truncated. 
\end{cor}
\begin{proof}
Combining Lemma \ref{mainequiv}, Lemma \ref{pointedsection}, and Corollary \ref{nondep}.
\end{proof}

\begin{cor}\label{deloop}
If $X$ is $n$-connected and $2n$-truncated, then $TX$ is the unique delooping of $X$.
\end{cor}
A different proof that such $X$ have unique deloopings is in~\cite[Theorem~6]{highergroups}.
\begin{proof}
In this case, the type of pointed $X$-torsors is $(-2)$-truncated, so
Theorem~\ref{thm} applies.
\end{proof}

\begin{cor}\label{deloopofh}
If $X$ is $n$-connected and $(2n+1)$-truncated and $TX$ is merely inhabited, then
$TX$ is the unique delooping of $X$.
\end{cor}
\begin{proof}
In this case the type of pointed $X$-torsors is $(-1)$-truncated, i.e. a
proposition. Since we assume $TX$ is merely inhabited, there also merely exists
a pointed $X$-torsor. A merely inhabited proposition is contractible, so we
can again apply Theorem~\ref{thm}.
\end{proof}

\section{Delooping maps}

Suppose $A, B : \UU_\pt$ are pointed types, and $f : \Omega A \to_\pt \Omega B$
is a pointed map on loop spaces. When can we find $F : A \to_\pt B$ such that
$f = \Omega F$? More precisely, we want a useful description of the type
$\Omega^{-1}f \coloneqq (F : A \to_\pt B) \times (f = \Omega F)$. For example,
it is necessary that we have $f(p \sq q) = f(p) \sq f(q)$.

\begin{lemma}\label{cab}
We have an equivalence of types
\[\Omega^{-1} f \simeq (a : |A|) \to (b : |B|) \times C\, a\, b\]
where $C : |A| \to |B| \to \UU$ is given by
\[C\, a\, b \coloneqq (h : (a = \pt_A) \to (b = \pt_B)) \times
	\+((p : a = \pt_A) \to f = D(h,p)\+)\]
and we define $D(h, p) : \Omega A \to_\pt \Omega B$ by
\[ |D(h, p)|(q) = (h\, p)^{-1} \sq\, h(p \sq q),\]
pointed in the obvious way.
\end{lemma}
We can think of $C$ as a proof-relevant relation approximating a function $F : |A| \to |B|$;
it would be a function if only $(b : |B|) \times C\, a\, b$ were contractible for all $a : |A|$.
\begin{proof}
We have
\begin{align*}
	\Omega^{-1} f &\simeq 
		(F : |A| \to |B|) \times (\pt_F : F\, \pt_A = \pt_B)) \times (f = \Omega(F, \pt_F))
	\\
	(a : |A|) \to (b : |B|) \times C\, a\, b
	&\simeq (F : |A| \to |B|) \times (a : |A|) \to C\, a\, (F\, a).
\end{align*}
So it suffices to show that for $F : |A| \to |B|$, we have
\[(\pt_F : F\, \pt_A = \pt_B) \times (f = \Omega (F, \pt_F)) 
	\simeq (a : |A|) \to C\, a\, (F\, a).\] Now by path induction and type-theoretic choice, we have
\begin{align*}(a : |A|) \to C\, a\, (F\, a) \simeq
(h : (a : |A|) \to (a = \pt_A) \to (F\, a = \pt_B))
 \times (f = D(h, \refl_{\pt_A})).
\end{align*}
Again by path induction, we have 
	$((a : |A|) \to (a = \pt_A) \to (F\, a = \pt_B)) \simeq (F\, \pt_A = \pt_B)$.
It suffices to show that if $h$ corresponds to $\pt_F$ under this equivalence, then
$D(h,\refl_{\pt_A}) = \Omega(F, \pt_F)$. This holds essentially by definition.
\end{proof}

\begin{cor}\label{maptrunc}
Suppose $|A|$ is $n$-connected and $|B|$ is $(2n+m+2)$-truncated, where $n \ge 0$ and $m \ge -2$.
Then $\Omega^{-1}f$ is $m$-truncated.
\end{cor}
\begin{proof}
It suffices to show that, for any $a : |A|$, the type $(b : |B|) \times C\, a\, b$ is $m$-truncated.
Since to be truncated is a proposition and $|A|$ is at least $0$-connected, it suffices
to consider the case where $a$ is $\pt_A$. In this case we have
\begin{align*}
(b : |B|) \times C\, \pt_A\, b &\simeq (b : |B|) \times ((h, t)
		: C\, \pt_A\, b) \times (q : b = \pt_B) \times (h\, \refl_{\pt_A} = q)\\
	&\simeq (h : \Omega A \to_\pt \Omega B) \times (p : \pt_A = \pt_A) \to (f = D(|h|, p)),
\end{align*}
by first adding two redundant singleton fields, and then removing another pair of singleton fields.
One can prove $E(h) : h = D(|h|, \refl_{\pt_A})$ using unit laws,
so we further have
\[ (b : |B|) \times C\, \pt_A\, b \simeq (t : (p : \pt_A = \pt_A) \to f = D(|f|,p))
	\times (t\, \refl_{\pt_A} = E(f)). \]
This is the type of pointed sections of a pointed type family over $\pt_A = \pt_A$.
The fibres are identity types in $\Omega A \to_\pt \Omega B$, which is $(n+m+1)$-truncated.
Since the fibres are $(n+m)$-truncated and the base $\pt_A = \pt_A$ is $(n-1)$-connected,
the type of pointed sections is $m$-truncated as claimed.
\end{proof}

\begin{cor}\label{elim}
If $|A|$ is $n$-connected and $|B|$ is $2n$-truncated, then $\Omega$ is an equivalence
\[ (A \to_\pt B) \simeq (\Omega A \to_\pt \Omega B). \]
\end{cor}

\begin{cor}\label{gpelim}
If $|A|$ is $n$-connected and $|B|$ is $(2n+1)$-truncated, then $\Omega$ identifies
$A \to_\pt B$ with the subtype of $\Omega A \to_\pt \Omega B$ consisting of
$f : \Omega A \to_\pt \Omega B$ such that $(b : |B|) \times C\, \pt_A\, b$,
which is logically equivalent to $C\, \pt_A\, \pt_B$,
and hence to $(p\, q : \pt_A = \pt_A) \to f(p \sq q) = f(p) \sq f(q)$.
\end{cor}

\section{Applications}

In homotopy type theory, we define the ordinary cohomology group $H^n(X; G)$ of
a type $X$ with coefficients in a an abelian group $G$ as the set-truncation 
$\trunc {X \to \K(G, n)}_0$, where $\K(G, n)$ is an Eilenberg--MacLane space.
The algebraic structure of these cohomology groups comes from various operations
at the level of Eilenberg--MacLane spaces, which we now discuss.

\subsection{$\boldsymbol{\K(G,n)}$}
Let $G$ be a group, so that in particular $G$ is a $0$-truncated type.
One can define a $0$-connected pointed type $\K(G, 1) : \UU_\pt$ with
$\Omega \K(G, 1) \simeq_\grp G$ as a type of torsors, similar to our $TX$ \cite{symmetry}.
Note that $\K(G, 1)$ is necessarily 1-truncated.
By Corollary \ref{gpelim}, we have that if $B$ is $1$-truncated, then
$(\K(G, 1) \to_\pt B) \simeq (G \to_\grp \Omega B)$; we think of this as an elimination
principle for $\K(G, 1)$. From this elimination principle, it follows that if
$X : \UU_\pt$ is another $0$-connected, $1$-truncated pointed type, then 
$(\K(G, 1) \simeq_\pt X) \simeq (G \simeq_\grp \Omega X)$.
					
When can we find $\K(G, 2) : \UU_\pt$
with $\Omega \K(G, 2) \simeq_\pt \K(G, 1)$? By Corollary \ref{deloopofh}, it suffices to
have 
	\[(\mu : (x : |\K(G, 1)|) \to \K(G, 1) \simeq_\pt (|\K(G, 1)|, x)) \times (\mu\, \pt = \id),\]
or equivalently 
	\[(\mu : (x : |\K(G, 1)|) \to G \cong_\grp (x = x)) \times (\mu\, \pt = \id).\]
Given a dependent elimination principle for $\K(G, 1)$, we could analyse this
type of pointed sections directly. Alternatively, we can think of pointed sections as 
pointed maps into a $\Sigma$-type with extra structure, and apply our non-dependent
elimination principle. The loop space of the $\Sigma$-type
($x : |\K(G, 1)|) \times G \cong_\grp (x = x)$ is the centre $\Zc(G)$ of $G$, and so we
are left to ask when the inclusion $\Zc(G) \to_\grp G$ has a section. This happens
precisely when $G$ is abelian. So $\K(G, 1)$ has a delooping if and only if $G$ is abelian, in
which case the delooping is unique. As soon as we have $\K(G, 2)$,
Corollary \ref{deloop} gives $\K(G, n) : \UU_\pt$ for every $n$
with $\Omega \K(G, n+1) \simeq_\pt \K(G, n)$. We also get an elimination
principle by repeated application of Corollary \ref{elim}: for any $n \ge 1$ and any 
$n$-truncated type $B$, we have $(\K(G, n) \to_\pt B) \simeq (G \to_\grp \Omega^n B)$.
One can check, combining the definition of $TX$ with the elimination principle for $\K(G, n)$,
that for $n \ge 0$ we have
\[
	\K(G, n+2) \simeq (Y : \UU) \times n\mathrm{-connected}(Y) \times (y : Y) \to 
		G \simeq_\grp \Omega^{n+1}(Y, y).
\]

\subsection{$\boldsymbol{\pi_n(S^n)}$}

While we have systematically avoided talking about higher inductive types, we
can still say something about them. Recall that the $n$-sphere $S^n : \UU_\pt$
is defined as a pointed type with $(S^n \to_\pt B) \simeq \Omega^n B$. If $B$
is $n$-truncated for $n \ge 1$, we have $\Omega^n B \simeq (\Z \to_\grp
\Omega^n B)$, since $\Z$ is the free group on one generator, which as
we've seen is equivalent to $\K(\Z, n) \to_\pt B$. By the Yoneda lemma, we get
that $\K(\Z, n)$ is the $n$-truncation of $S^n$. In particular, 
$\pi_n(S^n) \simeq_\grp \Omega^n(\trunc{S^n}_n) \simeq_\grp \Omega^n(\K(\Z,n)) \simeq_\grp \Z$,
and $\pi_k(S^n) = 0$ for $k < n$.

More generally, this argument shows that $A \simeq_\pt \trunc{\Sigma \Omega A}_{2n}$
when $A$ is $n$-connected and $2n$-truncated. Applying the same fact to the delooping $TA$
of $A$, we get that $TA \simeq_\pt \trunc{\Sigma A}_{2n+1}$. Taking loop spaces of both sides, we get $A \simeq_\pt \trunc{\Omega \Sigma A}_{2n}$, which is part of the
Freudenthal suspension theorem.

\subsection{The cup product}

We now give some sketches on how to define cohomology operations.
Given a bilinear map $L \to_\grp M \to_\grp N$, we define a cup-product
\[\smile\, : \K(L, n) \to_\pt \K(M, m) \to_\pt \K(N, n+m),\] 
similar to the definition in \cite{cohomology} and \cite[Definition~2.26]{hurewicz}. Note that we ask for the cup product
to respect pointing, corresponding to $0 \smile y = x \smile 0 = 0$; without this extra piece
of specification, the definition would not work.
Indeed, $\K(M, m) \to_\pt \K(N, n+m)$ is $n$-truncated,
so the elimination principle applies:%
\footnote{Formally this argument assumes $m, n \ge 1$, but it can be adapted to cover all
$m, n \ge 0$.}
\begin{align*}
\K(L,n) \to_\pt \K(M, m) \to_\pt \K(N, n+m)
		&\simeq L \to_\grp \Omega^n (\K(M, m) \to_\pt \K(N, n+m))\\
		&\simeq L \to_\grp \K(M, m) \to_\pt \Omega^n \K(N, m+n)\\
		&\simeq L \to_\grp \K(M, m) \to_\pt \K(N, m)\\
		&\simeq L \to_\grp M \to_\grp N.
\end{align*}
The forward maps in this composite are given explicitly by iterated looping,
so we arrive at a definition of the cup product as the unique bi-pointed map whose looping
gives back the bilinear map we started with. With this characterisation, we expect that 
algebraic properties of the cup product follow from analogous properties of looping. 
For example, one can prove that the following square
anti-commutes, and this corresponds to graded commutativity of the cup product.
\begin{figure}[!h]
\centering
\begin{tikzcd}
A \to_\pt B \to_\pt C 
	\arrow[r] \arrow[d] & 
\Omega A \to_\pt B \to_\pt \Omega C
\arrow[d]
\\
A \to_\pt \Omega B \to_\pt \Omega C 
	\arrow[r] &
\Omega A \to_\pt \Omega B \to_\pt \Omega^2 C
\end{tikzcd}
\end{figure}
\subsection{Steenrod squares}
Let us now use Corollary \ref{gpelim} to construct Steenrod squares as `stable cohomology operations'
$\Sq^i_n : \K(\Z/2, n) \to_\pt \K(\Z/2, n+i)$ with $\Sq^i_n$ corresponding to $\Omega \Sq^i_{n+1}$. 
We first define $\Sq^i_i$ as the
cup product square $x \mapsto x \smile x$. To deloop this to $\Sq^i_{i+1}$, 
we need to show \[(x + y) \smile (x+y) = x \smile x + y \smile y,\] which follows
from distributivity and graded commutativity
since we are working mod $2$. Given $\Sq^i_{i+1}$ we can define $\Sq^i_n$ for all
$n$ using Corollary \ref{elim}, by looping and delooping as appropriate.

\section{Concluding remarks}

Our Lemma \ref{cab} can be compared with the construction of functors out of a
Rezk completion in \cite[Theorem~9.9.4]{hottbook} and the construction of
maps $\K(G, 1) \to_\pt \K(H, 1)$ in \cite[Lemma~4.10.1]{symmetry}. Variants of
the relation $C\, a\, b$ are used in all cases. The idea can be understood as a
type-theoretic analogue of the arguments in \cite[Sections~5.2-5.3]{deligne}.

The arguments in this note are well-suited to formalisation. Indeed, many parts
have already been formalised twice: first by Louise
Leclerc \cite{louise}, and later by Axel Ljungström in order to
develop the theory of Steenrod squares.

In upcoming work, we take the ideas of this note much further to give an exact,
infinitary description of higher groups -- as well as higher
equivalence relations more generally -- and morphisms between them. In fact the description
of morphisms is in a precise sense obtained mechanically from the descriptions of
objects, explaining the similarity between the second and third sections of this note
(compare for example Corollaries \ref{tors}, \ref{deloop}, \ref{deloopofh} 
with Corollaries \ref{maptrunc}, \ref{elim}, \ref{gpelim}).

\printbibliography

\end{document}